\documentclass[11pt]{amsart}

\author{Carlo Sanna}
\address{Universit\`a degli Studi di Torino\\Department of Mathematics\\Turin, Italy}
\email{carlo.sanna.dev@gmail.com}

\title[On the asymptotic density of the support of a Dirichlet convolution]{\large On the asymptotic density of the support of a Dirichlet convolution}

\usepackage{amsmath}
\usepackage{amssymb}
\usepackage{amsthm}
\usepackage{geometry}
\usepackage{lgreek}
\usepackage{color}
\usepackage{hyperref}
\hypersetup{colorlinks=true}

\DeclareMathAlphabet{\curly}{U}{rsfs}{m}{n}

\newtheorem{thm}{Theorem}[section]
\newtheorem{cor}{Corollary}[section]
\newtheorem{lem}[thm]{Lemma}
\theoremstyle{remark}
\newtheorem{rmk}{Remark}[section]

\begin{document}

\renewcommand{\labelenumi}{(\roman{enumi})}
\def\sumprime{\sideset{}{^{'}}{\sum}}
\def\lcm{\mathrm{lcm}}
\def\N{\mathbf{N}}
\def\ds{\mathbf{d}}
\def\supp{\mathrm{supp}}

\begin{abstract}
Let $\nu$ be a multiplicative arithmetic function with support of positive asymptotic density.
We prove that for any not identically zero arithmetic function $f$ such that $\sum_{f(n) \neq 0} 1 / n < \infty$, the support of the Dirichlet convolution $f * \nu$ possesses a positive asymptotic density.
When $f$ is a multiplicative function, we give also a quantitative version of this claim.
This generalizes a previous result of P. Pollack and the author, concerning the support of M\"obius and Dirichlet transforms of arithmetic functions.
\end{abstract}

\maketitle

\section{Introduction}

Let $f$ and $g$ be two arithmetic functions, i.e., functions from the set of positive integers to the set of complex numbers.
The Dirichlet convolution of $f$ and $g$ is the arithmetic function denoted by $f * g$ and defined as
\begin{equation*}
(f * g)(n) := \sum_{d \mid n} f(d) \, g(n/d) ,
\end{equation*}
for all positive integers $n$. 
The set of arithmetic functions together with Dirichlet convolution forms a commutative monoid with identity element $\epsilon$, the arithmetic function that satisfies $\epsilon(1) = 1$ and $\epsilon(n) = 0$ for all integers $n \geq 2$.
Furthermore, any arithmetic function $f$ has an inverse $f^{-1}$ with respect to the Dirichlet convolution if and only if $f(1) \neq 0$, in which case $f^{-1}$ can be computed recursively by the identities $f^{-1}(1) = 1 / f(1)$ and
\begin{equation*}
f^{-1}(n) = - \frac1{f(1)} \sum_{\substack{d \mid n \\ d < n}} f^{-1}(d)\,f(n / d), \quad n \geq 2 .
\end{equation*}
The Dirichlet transform $\hat{f}$ and the M\"obius transform $\check{f}$ of the arithmetic function $f$ are defined by $\hat{f} := f * 1$ and $\check{f} := f * \mu$, where $\mu$ is the M\"obius function.
Notably, the Dirichlet inverse of $\mu$ is the identically equal to $1$ arithmetic function.
Actually, this is the content of the well-know M\"obius inversion formula, that is $\check{\hat{f}} = \hat{\check{f}} = f$ (see \cite[Ch. 2]{chan09} for details).

We call $(f,g)$ a M\"obius pair if $f$ and $g$ are arithmetic function with $f = \hat{g}$, or equivalently $g = \check{f}$.
In a previous paper, P. Pollack and the author studied the asymptotic density of the support of functions $f$ and $g$ in a M\"obius pair $(f,g)$; by the support of an arithmetic function $h$ we mean the set of all positive integers $n$ such that $h(n) \neq 0$, we denote it with  $\supp(h)$.
They give the following result \cite[Theorem 1.1]{PS13}.

\begin{thm}\label{old_theorem}
Suppose that $(f,g)$ is a nonzero M\"obius pair. 
If $\supp(f)$ is thin then $\supp(g)$ possesses a positive asymptotic density. 
The same result holds with the roles of $f$ and $g$ reversed.
\end{thm}

We recall that a set of positive integers $\curly{A}$ is said to be \emph{thin} if $\sum_{a \in \curly{A}} 1 / a < \infty$.
Given a multiplicative arithmetic function $\nu$, we extend the notion of M\"obius pair by saying that $(f,g)$ is a $\nu$-pair if $f$ and $g$ are arithmetic function with $g = f * \nu$, or equivalently $f = g * \nu^{-1}$. 
(Note that $\nu$ is Dirichlet invertible since as a multiplicative function it satisfies $\nu(1) = 1$).
Here, we prove the following generalization of Theorem \ref{old_theorem}.

\begin{thm}\label{main_theorem}
Let $\nu$ be a multiplicative arithmetic function with support of positive asymptotic density.
Suppose that $(f,g)$ is a nonzero $\nu$-pair. 
If $\supp(f)$ is thin then $\supp(g)$ possesses a positive asymptotic density.
\end{thm}

This is a true generalization of Theorem \ref{old_theorem}, since $1$ and $\mu$ are multiplicative arithmetic functions with support of positive asymptotic density.
Precisely, $\supp(\mu)$ is the set of squarefree numbers and it has density $6 / \pi^2$, as is well known.

The discovery of Theorem \ref{old_theorem} was initially motivated by the desire to prove a kind of \emph{uncertainty principle for the M\"obius transform}, in the sense that $f$ and $g$ cannot both be of thin support if $(f,g)$ is a nonzero M\"obius pair, which is in turn a generalization of a previous result of P. Pollack \cite{pollack11}.
Note the analogy with the well-know uncertainty principle of harmonic analysis, which states that a not identically zero function and its Fourier transform cannot both be compactly supported \cite{benedicks85}.

Now, in the same spirit, we show that Theorem \ref{main_theorem} leads to the following uncertainty principle.

\begin{cor}\label{uncert_cor}
Let $\nu$ be a multiplicative arithmetic function with support of positive asymptotic density.
If $(f,g)$ is a nonzero $\nu$-pair then $\supp(f)$ and $\supp(g)$ cannot be both thin.
\end{cor}

In the case when $f$ and $g$ are multiplicative arithmetic functions, we give also the following quantitative version of Theorem \ref{main_theorem}.

\begin{thm}\label{mul_bound}
Let $\nu$ be a multiplicative arithmetic function with support of positive asymptotic density.
Then there exists a constant $C_\nu > 0$, depending only on $\nu$, such that for every $\nu$-pair $(f,g)$ of multiplicative arithmetic functions it holds
\begin{equation*}
\ds(\supp(g)) \geq \frac{C_\nu}{\sum_{n \in \supp(f)} 1 / n} ,
\end{equation*}
where by convention $C_\nu / \infty := 0$. More specifically, we can choose $C_\nu$ as
\begin{equation*}
\frac{6}{\pi^2} \prod_{p \notin \supp(\nu)}\left(1 + \frac1{p}\right)^{-1} ,
\end{equation*}
where the product is proved to converges to a positive real number.
\end{thm}

P. Pollack and the author have also given a result similar to Theorem \ref{old_theorem}, but weighted by the absolute values of $f$ and $g$ \cite[Theorem 1.2]{PS13}.
We recall that the mean value of an arithmetic function $f$ is the limit $\lim_{x \to \infty} (1 / x)\sum_{n \leq x} f(n)$, whenever it exists.
\begin{thm}\label{old_theorem2}
Suppose that $(f,g)$ is a nonzero M\"obius pair. If 
\begin{equation*}
\sum_{n=1}^\infty \frac{|f(n)|}{n} < \infty ,
\end{equation*}
then $|g|$ possesses a positive mean value.
The same result holds with the roles of $f$ and $g$ reversed.
\end{thm}

We generalize Theorem \ref{old_theorem2} by proving the following:

\begin{thm}\label{second_theorem}
Let $\nu$ be a bounded multiplicative arithmetic function with support of positive asymptotic density.
Suppose that $(f,g)$ is a $\nu$-pair. If
\begin{equation*}
\sum_{n=1}^\infty \frac{|f(n)|}{n} < \infty ,
\end{equation*}
then $|g|$ possesses a finite mean value. Moreover, if $f$ do not vanish identically and
\begin{equation*}
\inf_{n \in \supp(\nu)} |\nu(n)| > 0 ,
\end{equation*}
then the mean value of $|g|$ is positive.
\end{thm}

\subsection*{Notation}
Hereafter, $\N$ denotes the set of positive integers and $\N_0 := \N \cup \{0\}$.
For any $\curly{A} \subseteq \N$ and $x \geq 0$ we write $\curly{A}(x) := \# (\curly{A} \cap [1,x])$ for the number of elements of $\curly{A}$ not exceeding $x$.
We use $\ds(\curly{A}) := \lim_{x \to \infty} \curly{A}(x) / x$ for the asymptotic density of $\curly{A}$, whenever this exists.
Similarly, we denote the lower and upper asymptotic density of $\curly{A}$ by $\underline{\ds}(\curly{A})$ and $\overline{\ds}(\curly{A})$, respectively.
The letter $p$ always denotes a prime number.
The notation $p^k  \mid\mid n$ means that $p^k \mid n$, but $p^{k+1} \nmid n$.

\section{Preliminaries}

This section is devoted to some lemmas needed for the proof of Theorem \ref{main_theorem}.
The first one deals with the existence of the asymptotic density of certain sieved sets of positive integers.

\begin{lem}\label{asym_dens_lem}
Let $\curly{B}$ be a set of positive integers and suppose that to every $b \in \curly{B}$ there corresponds a set $\Omega_b \subseteq \{0,1,\ldots,b-1\}$.
Furthermore, for any $b \in \curly{B}$ let
\begin{equation*}
\curly{S}_b := \{n \in \N : (n \bmod b) \in \Omega_b \}
\end{equation*}
and assume that $\curly{S}_b(x) \leq c_b x$ for all $x \geq 0$, where $c_b$ are positive constants
satisfying 
\begin{equation*}
\sum_{b \in \curly{B}} c_b < \infty .
\end{equation*}
Then the set of positive integers $n$ such that $(n \bmod b) \not\in \Omega_b$ for all $b \in \curly{B}$ possesses an asymptotic density.
\end{lem}
\begin{proof}
We have to prove that the set
\begin{equation*}
\curly{S}_{\curly{B}} := \{n \in \N : (n \bmod b) \notin \Omega_b \quad \forall b \in B\} ,
\end{equation*}
has an asymptotic density.
If $\curly{B}$ is empty then $\curly{S}_{\curly{B}} = \N$ and the claim is trivial.

Suppose that $\curly{B}$ is finite with $k \geq 1$ elements $b_1, b_2, \ldots, b_k$.
From the inclusion-exclusion principle it follows
\begin{equation}\label{incl_excl_eq}
\curly{S}_{\curly{B}}(x) = \N(x) + \sum_{h=1}^k (-1)^h \sum_{1 \leq i_1 < \cdots < i_h \leq k} \big(\curly{S}_{b_{i_1}} \cap \cdots \cap \curly{S}_{b_{i_h}}\big)(x) ,
\end{equation}
for all $x > 0$. 
Now, fix $h$ and $i_1, \ldots, i_h$ positive integers such that $h \le k$, $i_1 < \cdots < i_h \leq k$, as in equation (\ref{incl_excl_eq}), and set $\ell := \lcm(b_{i_1}, b_{i_2}, \ldots, b_{i_h})$.
As a consequence of the Chinese Remainder Theorem, there exists a subset $\Theta$ of $\{0,1,\ldots,\ell-1\}$ such that for $n \in \N$ it holds $n \in \curly{S}_{b_{i_1}} \cap \cdots \cap \curly{S}_{b_{i_h}}$ if and only if $(n \bmod \ell) \in \Theta$.
So $\curly{S}_{b_{i_1}} \cap \cdots \cap \curly{S}_{b_{i_h}}$ has asymptotic density $\# \Theta / \ell$.
Dividing (\ref{incl_excl_eq}) by $x$ and letting $x \to \infty$ yields that $\curly{S}_\curly{B}$ has an asymptotic density.

Suppose now that $\curly{B}$ is infinite and let $b_1, b_2, \ldots$ be a numbering of $\curly{B}$.
For all positive integer $k$ define $\curly{B}_k := \{b_1, b_2, \ldots, b_k\}$. 
We have just seen that $\curly{S}_{\curly{B}_k}$ has an asymptotic density, so put $d_k := \ds(\curly{S}_{\curly{B}_k})$.
Since $\curly{S}_{\curly{B}_1} \supseteq \curly{S}_{\curly{B}_2} \supseteq \cdots$, it follows that $d_1, d_2, \ldots$ is a nonnegative decreasing sequence, so there exists $d := \lim_{k \to \infty} d_k$.
Furthermore
\begin{equation}\label{approx_eq}
0 \leq \curly{S}_{\curly{B}_k}(x) - \curly{S}_{\curly{B}}(x) \leq \sum_{i=k+1}^\infty \curly{S}_{b_i}(x) \leq x \sum_{i=k+1}^\infty c_{b_i}
\end{equation}
for all $x > 0$ and $k \in \N$, 
where the last term in the above inequality is a convergent series (by the hypothesis).
Dividing equation (\ref{approx_eq}) by $x$ and letting $x \to \infty$ we obtain that
\begin{equation}\label{approx_eq2}
\limsup_{x \to \infty}\left| \frac{\curly{S}_{\curly{B}}(x)}{x} - d_k\right| \leq \sum_{i=k+1}^\infty c_{b_i} .
\end{equation}
Finally, letting $k \to \infty$ in equation (\ref{approx_eq2}) it follows that $\curly{S}_{\curly{B}}$ has asymptotic density $d$.
\end{proof}

If $\curly{A}$ is a set of positive integers, we write $\curly{M}(\curly{A}) := \{ an : a \in \curly{A},\; n \in \N\}$ for the \emph{set of multiples} of $\curly{A}$.
The interested reader can found many results on sets of multiples in \cite{hall96}.
We need only to state the following lemma about the asymptotic density of $\curly{M}(\curly{A})$.

\begin{lem}\label{set_of_mul_lem}
If $\curly{A}$ is a thin set of positive integers, then $\curly{M}(\curly{A})$ has an asymptotic density.
Moreover, if $1 \notin \curly{A}$ then $\ds(\curly{M}(\curly{A})) < 1$.
\end{lem}
\begin{proof}
See \cite[Lemma 2.2]{PS13}. Incidentally, note that the existence of $\ds(\curly{M}(\curly{A}))$ is a corollary of Lemma \ref{asym_dens_lem}, setting $\curly{B} := \curly{A}$, $\Omega_b = \{0\}$ and $c_b := 1 / b$ for all $b \in \curly{B}$.
\end{proof}

\begin{lem}\label{asym_dens_lem2}
Let $\curly{A}$ be a set of positive integers, $\curly{P}$ a set of prime numbers, and suppose that to every $p \in \curly{P}$ there corresponds a set of nonnegative integers $\curly{K}_p$.
Define $\curly{S}$ to be the set of all positive integers $n$ such that neither $a \mid n$ for some $a \in \curly{A}$, nor $p^k \mid\mid n$ for some $p \in \curly{P}$ and $k \in \curly{K}_p$.
If $\curly{P}_1 := \{p \in \curly{P} : 1 \in \curly{K}_p\}$ and $\curly{A} \cup \curly{P}_1$ if thin then $\curly{S}$ possesses an asymptotic density.
Moreover, if $0 \notin \curly{K}_p$ for all $p \in \curly{P}$ and $1 \notin \curly{A}$ then $\ds(\curly{S}) > 0$.
\end{lem}
\begin{proof}
Consider first the case when $0 \notin \curly{K}_p$ for all $p \in \curly{P}$.
We want to use Lemma \ref{asym_dens_lem}. 
For define $\curly{B} := \curly{A} \cup \big\{p^{k+1} : p \in \curly{P},\; k \in \curly{K}_p\big\}$. 
We construct $\Omega_b$ as follows: start with empty $\Omega_b$ for all $b \in \curly{B}$, then throw $0$ into $\Omega_a$ for all $a \in \curly{A}$ and throw $p^k, 2p^k, \ldots, (p-1)p^k$ into $\Omega_{p^{k+1}}$ for all $p \in \curly{P}$ and $k \in \curly{K}_p$.
If $a \in \curly{A} \setminus \big\{p^{k+1} : p \in \curly{P},\; k \in \curly{K}_p\big\}$ then
$\curly{S}_{a}(x) \leq x / a$, so set $c_a := 1 / a$.
On the other hand, if $p \in \curly{P}$ and $k \in \curly{K}_p$ then $\curly{S}_{p^{k+1}}(x) \leq x / p^k$, so set $c_{p^{k+1}} := 1 / p^k$.
Since $\curly{A} \cup \curly{P}_1$ is thin, it follows
\begin{align*}
\sum_{b \in \curly{B}} c_b &\leq \sum_{a \in \curly{A}} \frac1{a} + \sum_{p \in \curly{P}_1} \sum_{k \in \curly{K}_p} \frac1{p^k} + \sum_{p \in \curly{P} \setminus \curly{P}_1} \sum_{k \in \curly{K}_p} \frac1{p^k} \\
&\leq \sum_{a \in \curly{A}} \frac1{a} + \sum_{p \in \curly{P}_1} \sum_{k = 1}^\infty \frac1{p^k} + \sum_{p \in \curly{P} \setminus \curly{P}_1} \sum_{k = 2}^\infty \frac1{p^k} \\
&\leq \sum_{a \in \curly{A}} \frac1{a} + \sum_{p \in \curly{P}_1} \frac1{p-1} + \sum_{p} \frac1{p^2-p} < \infty ,
\end{align*}
and thus Lemma \ref{asym_dens_lem} implies that $\curly{S}$ has an asymptotic density, since $n \in \curly{S}$ if and only if $(n \bmod b) \notin \Omega_b$ for all $b \in \curly{B}$.
In particular, if $1 \notin \curly{A}$, then the set $\curly{C} := \curly{A} \cup \curly{P}_1 \cup \{p^2 : p \in \curly{P}\}$ is thin and also $1 \notin \curly{C}$.
Hence, it follows from Lemma \ref{set_of_mul_lem} that $\curly{M}(\curly{C})$ has asymptotic density less than $1$.
On the other hand, $(\N \setminus \curly{M}(\curly{C})) \subseteq \curly{S}$, so $\ds(\curly{S}) > 0$.

Now consider the case when $0 \in \curly{K}_p$ for some $p \in \curly{P}$ and define $\curly{P}_0 := \{p \in \curly{P} : 0 \in \curly{K}_p\}$.
It results that $p \mid n$ for all $p \in \curly{P}_0$ and $n \in \curly{S}$.
If $\curly{P}_0$ is infinite, then $\curly{S}$ is empty, and hence has asymptotic density zero.
On another hand, if $\N \subseteq \curly{K}_p$ for some $p \in \curly{P}_0$ then $\curly{S}$ is empty again and the claim follows.
So we are left with the case when $\curly{P}_0$ is finite and for each $p \in \curly{P}_0$ there exists a positive integer $k_p$ such that $\{0,1,\ldots,k_p-1\} \subseteq \curly{K}_p$, but $k_p \notin \curly{K}_p$.
Therefore, any $n \in \curly{S}$ is divisible by $\pi := \prod_{p \in \curly{P}_0} p^{k_p}$.
Define
\begin{equation*}
\curly{A}^\prime := \left\{\frac{a}{\gcd(a,\pi)} : a \in \curly{A} \right\} ,
\end{equation*}
and let $\curly{K}_p^\prime := \{k - k_p : k \in \curly{K}_p, \; k > k_p \}$ for $p \in \curly{P}_0$ and $\curly{K}_p^\prime := \curly{K}_p$ for $p \in \curly{P} \setminus \curly{P}_0$.
Then $n \in \curly{S}$ if and only if $n = \pi m$ for a positive integer $m$ such that neither $a \mid m$ for some $a \in \curly{A}^\prime$, nor $p^k \mid\mid m$ for some $p \in \curly{P}$ and $k \in \curly{K}_p^\prime$.
Note that $\curly{A}^\prime$ is thin since $\curly{A}$ is thin.
Furthermore, let $\curly{P}_1^\prime := \{p \in \curly{P} : 1 \in \curly{K}_p^\prime\}$, then $\curly{P}_1^\prime \subseteq \curly{P}_0 \cup \curly{P}_1$, so $\curly{P}_1^\prime$ is thin, since $\curly{P}_0$ is finite and $\curly{P}_1$ is thin by hypothesis.
This yields that $\curly{A}^\prime \cup \curly{P}_1^\prime$ is itself thin.
Since $0 \notin \curly{K}_p^\prime$ for all $p \in \curly{P}$, it then follows from the first part of the proof that $\curly{S}$ possesses an asymptotic density.
\end{proof}

Now we show that the support of any multiplicative arithmetic function has an asymptotic density and we give a way to know if this density is zero or positive.
This is particularly useful if one needs to apply Theorem \ref{main_theorem}.

\begin{lem}\label{mul_fun_charac}
If $\nu$ is a multiplicative arithmetic function then $\supp(\nu)$ has an asymptotic density and specifically
\begin{equation}\label{mul_dens}
\ds(\supp(\nu)) = \prod_{p} \left(1 - \frac1{p}\right) \!\!\sum_{\substack{k=0 \\ p^k \in \supp(\nu)}}^\infty \frac1{p^k} .
\end{equation}
In particular, $\ds(\supp(\nu)) > 0$ if and only if $\sum_{p \notin \supp(\nu)} 1/p < \infty$.
\end{lem}
\begin{proof}
From a result of G. Tenenbaum \cite[Theorem 11, p. 48]{tenenbaum95}, if $f$ is a multiplicative arithmetic function with values in $[0,1]$ then
\begin{equation*}
\lim_{x \to \infty} \frac1{x} \sum_{n \leq x} f(n) = \prod_{p} \left(1 - \frac1{p}\right) \sum_{k=0}^\infty \frac{f(p^k)}{p^k} .
\end{equation*}
Equation (\ref{mul_dens}) follows choosing $f$ as the indicator function of $\supp(\nu)$, which is multiplicative.
After some calculations, we obtain
\begin{equation*}
\ds(\supp(\nu)) = \prod_{p \in \supp(\nu)} \left(1 - \frac{c_p}{p^2}\right) \prod_{p \notin \supp(\nu)} \left(1 - \frac1{p} + \frac{c_p}{p^2}\right) ,
\end{equation*}
where $c_p \in [0,1]$ for all prime numbers $p$. 
In conclusion, $\ds(\supp(\nu)) > 0$ if and only if $\sum_{p \notin \supp(\nu)} 1/p < \infty$.
Regarding the convergence of infinite products, see \cite[Ch. 8]{krantz08}.
\end{proof}

The next lemma is the key to the proof of Theorem \ref{main_theorem}. It is a generalization of \cite[Lemma 2.3]{PS13}, hence the first parts of their proof are similar.

\begin{lem}\label{powered_lem}
Let $\nu$ be a multiplicative arithmetic function with support of positive asymptotic density.
Let $\curly{A}$ be a thin set of positive integers. If $\curly{T} \subseteq \curly{S} \subseteq \curly{A}$, where $\curly{S}$ is finite, then the set $\curly{C}$ of positive integers $n$ for which both:
\begin{enumerate}
\item $\curly{S} = \{d \in \curly{A} : d \mid n\}$; and
\item $\curly{T} = \{d \in \curly{S} : d \mid n, \; \nu(n/d) \neq 0\}$,
\end{enumerate}
has an asymptotic density.
\end{lem}
\begin{proof}
Let $\chi_\nu$ be the indicator function of $\supp(\nu)$.
Define the arithmetic function $\chi$ by taking
\begin{equation}\label{indic_eq}
\chi(n) := \prod_{d \in \curly{T}} \chi_\nu(n/d) \prod_{c \in \curly{S} \setminus \curly{T}}(1-\chi_\nu(n/c)) ,
\end{equation}
for each $n$ satisfying condition (i), and let $\chi(n) := 0$ otherwise.
Then $\chi$ is the indicator function of $\curly{C}$. 
Moreover, when $n$ satisfies (i), expanding the second product in \eqref{indic_eq} we obtain that
\begin{equation*}
\chi(n) = \sum_{\curly{T} \subseteq \curly{U} \subseteq \curly{S}} (-1)^{|\curly{U}|-|\curly{T}|} \prod_{e \in \curly{U}} \chi_\nu(n/e) .
\end{equation*}
So using $'$ to denote a sum restricted to integers $n$ satisfying (i), we find that
\begin{align}\label{count_eq}
\curly{C}(x) = \sum_{n \leq x} \chi(n) &=  \sum_{\curly{T} \subseteq \curly{U} \subseteq \curly{S}} (-1)^{|\curly{U}|-|\curly{T}|} \sumprime_{n \leq x} \prod_{e \in \curly{U}} \chi_\nu(n/e) ,
\end{align}
for all $x > 0$.
Dividing equation (\ref{count_eq}) by $x$ and letting $x \to \infty$, it suffices to prove that for each set $\curly{U}$ with $\curly{T} \subseteq \curly{U} \subseteq \curly{S}$, the set
\begin{equation*}
\curly{V} := \{n \in \N : \text{$n$ satisfies (i), $n / e \in \supp(\nu)$ for all $e \in \curly{U}$}\} 
\end{equation*}
has an asymptotic density.

If $n$ satisfies (i) then $L \mid n$, where $L := \lcm\{d \in \curly{S}\}$.
In fact, (i) holds for $n$ if and only if $n = Lq$ for some $q \in \N$ such that $a / \gcd(a, L) \nmid q$ for all $a \in \curly{A} \setminus \curly{S}$.
On the other hand, $n / e \in \supp(\nu)$ for $e \in \curly{U}$ if and only if there must exist no prime $p$ and positive integer $k$ such that $p^k \mid\mid n / e$ and $\nu(p^k) = 0$.
Now, in view of using Lemma \ref{asym_dens_lem2}, define
\begin{equation*}
\curly{B} := \left\{\frac{a}{\gcd(a, L)} : a \in \curly{A} \setminus \curly{S} \right\} ,
\end{equation*}
and let $\curly{P}$ be the set of prime numbers $p$ such that $\nu(p^k) = 0$ for some $k \in \N$.
Being that $\curly{A}$ is a thin set it follows that also $\curly{B}$ is a thin set.
Moreover, define
\begin{equation*}
\curly{K}_p := \big\{ h - j : h \in \N, \, j \in \N_0, \, e \in \curly{U}, \, \nu(p^h) = 0, \, p^j \mid\mid L / e \big\} \cap \N_0 .
\end{equation*}
So, it follows that $n = Lq \in \curly{V}$ if and only if neither $b \mid q$ for some $b \in \curly{B}$, nor $p^k \mid\mid q$ for some $p \in \curly{P}$ and $k \in \curly{K}_p$.
Let $\curly{P}_1 := \{p \in \curly{P} : 1 \in \curly{K}_p\}$. If $p \in \curly{P}_1$ then there are only two possible cases: $p \notin \supp(\nu)$ and $p \nmid L / e$ for some $e \in \curly{U}$; or $\nu(p^{j+1}) = 0$ and $p^j \mid \mid L / e$ for some $j \in \N$, $e \in \curly{U}$ (there are only a finite number of such primes, because $\curly{U}$ is finite).
Since $\nu$ has support of positive asymptotic density, from Lemma \ref{mul_fun_charac} it results that $\sum_{p \notin \supp(\nu)} 1 / p < \infty$, so by the previous consideration $\curly{P}_1$ is a thin set.
In conclusion, $\curly{B} \cup \curly{P}_1$ is thin and from Lemma \ref{asym_dens_lem2} it follows that $\curly{V}$ possesses an asymptotic density.
This completes the proof.
\end{proof}

\section{Proofs of Theorem \ref{main_theorem} and Corollary \ref{uncert_cor}}

The idea in the proof of Theorem \ref{main_theorem} is the same as for Theorem \ref{old_theorem}, but the claim is now strengthened and the proof is simplified by Lemma \ref{asym_dens_lem2}.
In particular, it is no longer necessary, with the new approach, to split the proof into two parts, as was done for Theorem \ref{old_theorem}.

We call two elements $n_1$ and $n_2$ of $\supp(g)$ equivalent if they share the same set $\curly{S}$ of divisors from $\supp(f)$ and for $d \in \curly{S}$ it holds $\nu(n_1 / d) \neq 0$ if and only if $\nu(n_2 / d) \neq 0$.
Actually, this is an equivalence relation, with equivalence classes $\curly{A}_1, \curly{A}_2, \ldots$
Then to any $\curly{A}_i$ there correspond a nonempty set $\curly{S}_i$ and a subset $\curly{T}_i$ of $\curly{S}_i$ such that $n \in \curly{A}_i$ if and only if $\{d \in \supp(f) : d \mid n\} = \curly{S}_i$ and $\{d \in \curly{S}_i : \nu(n / d) \neq 0\} = \curly{T}_i$.
Moreover, $\curly{T}_i \subseteq \curly{S}_i \subseteq \supp(f)$ with $\supp(f)$ thin by hypotheses and $\curly{S}_i$ finite, since it is a set of divisors of a positive integer.
It follows from Lemma \ref{powered_lem} that $\curly{A}_i$ possesses an asymptotic density.
On the other hand, 
\begin{equation*}
\supp(g) = \bigcup_{i=1}^\infty \curly{A}_i ,
\end{equation*}
with disjoint union.
If there are only finitely many $\curly{A}_i$ then it follows immediately that $\ds(\supp(g))$ exists, since the asymptotic density if finitely additive.
If instead there are infinitely many $\curly{A}_i$ then define 
\begin{equation*}
m_k := \min_{i > k} \max(\curly{S}_i)
\end{equation*}
and observe that $m_k \to \infty$ as $k \to \infty$.
If $n \in \bigcup_{i > k} \curly{A}_i$ then $\{d \in \supp(f) : d \mid n\} = \curly{S}_i$ for some integer $i > k$.
So $n$ has a divisor $d \in \supp(f)$ with $d \geq m_k$ and as a consequence
\begin{equation}\label{sup_ds_bound}
\overline{\ds}\!\!\left(\bigcup_{i > k} \curly{A}_i\right) \leq \sum_{\substack{d \in \supp(f) \\ d \geq m_k}} \frac1{d} .
\end{equation}
As $k \to \infty$ the right-hand side of equation (\ref{sup_ds_bound}) tends to zero, since $\supp(f)$ is thin.
Therefore, it follows that $\supp(g)$ has an asymptotic density (see \cite[Lemma 2.1]{PS13}).

Now, we want to prove that $\ds(\supp(g)) > 0$.
Since $f$ does not vanish identically, $\supp(f)$ has a minimum $d$.
We claim that a positive portion of positive integers $n$
satisfies the following conditions: $n$ has only $d$ as a divisor from $\supp(f)$ and $n / d \in \supp(\nu)$; to the effect that $g(n) = f(d) \,\nu(n / d) \neq 0$, i.e., $n \in \supp(g)$.
Define
\begin{equation*}
\curly{B} := \left\{\frac{b}{\gcd(b,d)} : b \in \supp(f), \, b \neq d \right\} ,
\end{equation*}
and let $\curly{P}$ be the set of prime numbers $p$ such that $\nu(p^k) = 0$ for some $k \in \N$, for each $p \in \curly{P}$ set $\curly{K}_p := \{k \in \N : \nu(p^k) = 0\}$.
Let $\curly{V}$ be the set of positive integers $m$ such that neither $b \mid m$ for some $b \in \curly{B}$, nor $p^k \mid\mid m$ for some $p \in \curly{P}$ and $k \in \N$.
If $m \in \curly{V}$ then $n = dm$ has $d$ as his only divisor from $\supp(f)$ and $n / d = m \in \supp(\nu)$.
Note that $\curly{B}$ is a thin set, since $\supp(f)$ is thin. 
Thanks to Lemma \ref{mul_fun_charac} it results that $\curly{P}_1 := \{p \in \curly{P} : 1 \in \curly{K}_p\}$ is thin, so $\curly{B} \cup \curly{P}_1$ is thin.
Moreover, $0 \notin \curly{K}_p$ for all $p \in \curly{P}$ and $1 \notin \curly{B}$, so it follows from Lemma \ref{asym_dens_lem2} that $\ds(\curly{V}) > 0$ and finally $\ds(\supp(g)) > 0$.

At this point, the proof of Corollary \ref{uncert_cor} is immediate.
By partial summation one can show that a set of positive asymptotic density is never thin.
Let $(f,g)$ be a nonzero $\nu$-pair.
If $\supp(f)$ is thin then from Theorem \ref{main_theorem} it follows that $\supp(g)$ has positive asymptotic density and hence it is not thin.
Otherwise, if $\supp(g)$ is thin, note that $\nu^{-1}$ is a multiplicative function because it is the Dirichlet inverse of a multiplicative function.
Furthermore, $\nu^{-1}(p) = -\nu(p)$ for all primes $p$.
Thus, we get by Lemma \ref{mul_fun_charac}
\begin{equation*}
\sum_{p \notin \supp(\nu^{-1})} \frac1{p} = \sum_{p \notin \supp(\nu)} \frac1{p} < \infty ,
\end{equation*}
and hence $\supp(\nu^{-1})$ has positive asymptotic density.
From Theorem \ref{main_theorem}, since $(g,f)$ is a $v^{-1}$-pair, it follows that $\supp(f)$ has positive asymptotic density and so it is not thin.

\section{Proof of Theorem \ref{mul_bound}}

Since $\nu$ is a multiplicative arithmetic function with support of positive asymptotic density, it follows from Lemma \ref{mul_fun_charac} that $\sum_{p \notin \supp(\nu)} 1 / p < \infty$, and so
\begin{equation*}
C_\nu^\prime := \frac{6}{\pi^2} \prod_{p \notin \supp(\nu)} \left(1 + \frac1{p}\right)^{-1}
\end{equation*}
is a well-defined positive real constant.
On the one hand, since $f$ is multiplicative, it is easily seen that
\begin{equation*}
\sum_{n \in \supp(f)} \frac1{n} \geq \prod_{p \in \supp(f)} \left(1 + \frac1{p}\right) .
\end{equation*}
On the other hand, since $g$ is multiplicative too, we get, again by Lemma \ref{mul_fun_charac}, that
\begin{align*}
\ds(\supp(g)) &= \prod_{p} \left(1 - \frac1{p}\right) \!\! \sum_{\substack{k=0 \\ p^k \in \supp(g)}}^\infty \frac1{p^k} \geq \prod_{p \in \supp(g)} \left(1 - \frac1{p^2}\right) \prod_{p \notin \supp(g)} \left(1 - \frac1{p}\right) \\
&= \prod_{p} \left(1 - \frac1{p^2}\right) \prod_{p \notin \supp(g)} \left(1 + \frac1{p}\right)^{-1} = \frac{6}{\pi^2} \prod_{p \notin \supp(g)} \left(1 + \frac1{p}\right)^{-1} .
\end{align*}
Finally, $g(p) = \nu(p) + f(p)$ for all prime numbers $p$, so we obtain
\begin{align*}
\ds(\supp(g)) &= \frac{6}{\pi^2} \prod_{\substack{p \notin \supp(\nu) \\ p \notin \supp(f)}} \left(1 + \frac1{p}\right)^{-1} \prod_{\substack{p \in \supp(f) \\ f(p) \neq -\nu(p)}} \left(1 + \frac1{p}\right)^{-1} \\
&\geq C_\nu^\prime \prod_{p \in \supp(f)} \left(1 + \frac1{p}\right)^{-1} \geq \frac{C_\nu^\prime}{\sum_{n \in \supp(f)} 1 / n} .
\end{align*}
This completes the proof.

\begin{rmk}
Theorem \ref{mul_bound} 
is no longer true if the hypothesis that $f$ is multiplicative, or equivalently that $g$ is multiplicative, is dropped.
For example, fix an integer $d \geq 2$ and take $f$ as the indicator function of the singleton $\{d\}$, to the effect that $f$ is not multiplicative.
On the one hand, it results $\ds(\supp(g)) \leq 1 / d$.
On the other hand, obviously $\sum_{n \in \supp(f)} 1 / n = 1 / d$.
Thus, it must be $C_\nu \leq 1 / d^2$.
Due to the arbitrariness of $d$ if follows that $C_\nu$ cannot be positive.
\end{rmk}

An interesting question might be the evaluation of the best constant $C_\nu$ in Theorem \ref{mul_bound}, i.e., the infimum
\begin{equation*}
C_\nu := \inf_{(f,g)} \; \ds(\supp(g)) \sum_{n \in \supp(f)} \frac1{n} ,
\end{equation*}
over all $\nu$-pairs $(f,g)$ of multiplicative arithmetic functions with $f$ of thin support.
Theorem \ref{mul_bound} gives us a lower bound for $C_\nu$.
For an upper bound, notice that setting $f = \epsilon$ we obtain $C_\nu \leq \ds(\supp(\nu))$.
Thus, in particular, we have $C_\mu = 6 / \pi^2$.

\section{Proof of Theorem \ref{second_theorem}}

For each $y \geq 0$ define the function $g_y$ by setting
\begin{equation*}
g_y(n) := \sum_{\substack{d \mid n \\ d \leq y}} f(d) \, \nu(n / d) ,
\end{equation*}
for all $n \in \N$. 
We can regard $g_y$ as a sort of ``truncated Dirichlet convolution'' of $f$ and $\nu$.
The following lemma holds.

\begin{lem}
If $\sum_{n=1}^\infty |f(n)| / n < \infty$ then:
\begin{enumerate}
\item For all $y \geq 0$ the function $|g_y|$ has a finite mean value $\lambda_y$.
\item $\lambda_y$ tends to a finite limit $\lambda$ as $y \to \infty$.
\item $|g|$ has mean value $\lambda$.
\end{enumerate}
\end{lem}
\begin{proof}
The proof is almost identical to that of \cite[Lemma 4.1]{PS13}, so we do not give the details. 
The only differences is that one needs to use Lemma \ref{powered_lem} instead of \cite[Lemma 2.3]{PS13}
and that in the proof of (ii) and (iii) one makes use of the boundedness of $\nu$.
\end{proof}

Now, suppose that $\delta := \inf_{n \in \supp(\nu)} |\nu(n)|$ is positive and $f$ is not identically zero.
We want to prove that the mean value of $|g|$ is positive.
Let $d$ be the least positive integer in $\supp(f)$.
In the proof of Theorem \ref{main_theorem} we have seen that a positive portion of $n \in \N$ has $d$ as their only divisor from $\supp(f)$ and satisfies $n / d \in \supp(\nu)$, so that $|g(n)| = |f(d)|\,|\nu(n/d)| \geq |f(d)| \,\delta$.
Let $\curly{A}$ denote the set of these integers $n$, then
\begin{equation*}
\frac1{x} \sum_{n \leq x} |g(n)| \geq \frac1{x} \sum_{\substack{n \leq x \\ n \in \curly{A}}} |g(n)| \geq |f(d)| \,\delta \, \frac{\curly{A}(x)}{x} > 0
\end{equation*}
for large $x$, since $\underline{\ds}(\curly{A}) > 0$.
Hence, the mean value of $|g|$ is positive.
This completes the proof.

\begin{rmk}
In Theorem \ref{second_theorem}, the existence of the mean value of $|g|$ is no longer guaranteed if the hypothesis of boundedness of $\nu$ is omitted.
For example, consider the arithmetic functions $\nu$ defined by $\nu(n) := n$ for all $n \in \N$ and $f = \epsilon$; it results that $|g|$ has not a finite mean value.
Furthermore, the positiveness of the mean value of $|g|$ is no longer guaranteed if the hypothesis $\inf_{n \in \supp(\nu)} |\nu(n)| > 0$ is omitted.
E.g., consider the arithmetic functions $\nu$ defined by $\nu(n) := 1/n$ for all $n \in \N$ and $f = \epsilon$, it results that $|g|$ has mean value zero.
\end{rmk}

\subsection*{Acknowledgements}
The author thanks Salvatore Tringali (LJLL, Universit\'e Pierre et Marie Curie) for his helpful proofreading and suggestions.

\bibliographystyle{amsalpha}
\providecommand{\bysame}{\leavevmode\hbox to3em{\hrulefill}\thinspace}
\providecommand{\MR}{\relax\ifhmode\unskip\space\fi MR }

\providecommand{\MRhref}[2]{
  \href{http://www.ams.org/mathscinet-getitem?mr=#1}{#2}
}
\providecommand{\href}[2]{#2}

\end{document}